\documentclass[9pt]{article}
\usepackage{amsmath, amsfonts, amsthm, amssymb, verbatim,dsfont}
\usepackage{graphicx}
\usepackage[mathcal]{euscript}
\usepackage{amstext} 
\usepackage{amssymb,mathrsfs}
\theoremstyle{plain}
\newtheorem{theorem}{Theorem}[section]
\newtheorem{proposition}[theorem]{Proposition}
\newtheorem{lemma}[theorem]{Lemma}

\newtheorem{definition}[theorem]{Definition}
\theoremstyle{definition}
\newtheorem{remark}[theorem]{Remark}
\theoremstyle{plain}

\numberwithin{equation}{section}
\newcommand \loc 	{\text{loc}} 
\newcommand \be    	{\begin{equation}}
\newcommand \ee     {\end{equation}} 
\newcommand \RR      {\mathbb{R}} 
\newcommand \del     \partial 
\DeclareMathOperator  \sgn {sgn} 
\DeclareMathOperator   \TV  {TV} 
\newcommand \nn     {\mathbf n} 
\newcommand \TT     {\mathcal{T}}
\newcommand \Tcal   {\TT}
\newcommand \dK		{\partial K}
\newcommand \nek		{\nn_{e,K}} 
\newcommand \fek		{f_{e,K}}
\newcommand \feke	{f_{e,K_e}}
\newcommand \Fek		{F_{e,K}} 
\newcommand \Lip		{\mathop{\mathrm{Lip}}}
\newcommand \unk		{u^n_K}
\newcommand \unke	{u^n_{K_e}} 
\newcommand \unnkee	{u^{n+1}_{K,e}} 
\newcommand \tunnkee		{\tilde u^{n+1}_{K,e}} 
\newcommand \dnnkee	{D^{n+1}_{K,e}}
\newcommand \Hcal	{\mathcal{H}}
\def\Xint#1{\mathchoice {\XXint\displaystyle\textstyle{#1}}%
{\XXint\textstyle\scriptstyle{#1}}%
{\XXint\scriptstyle\scriptscriptstyle{#1}}%
{\XXint\scriptscriptstyle\scriptscriptstyle{#1}}%
\!\int}
\def\XXint#1#2#3{{\setbox0=\hbox{$#1{#2#3}{\int}$}
\vcenter{\hbox{$#2#3$}}\kern-.5\wd0}}

\def\dashint{\Xint\diagup}
\DeclareMathOperator\divex {div}
\newcommand{\dive}{\divex_{g}}
\DeclareMathOperator\tr {tr} \DeclareMathOperator \ct {C}
\DeclareMathOperator\e {e}
\newcommand \Euhu	{E_{\delta, \epsilon}^h(u^h,u)}
\newcommand \Euuh	{E_{\delta, \epsilon}^h(u,u^h)}
\newcommand \Thuhu	{\Theta_{ \delta,\epsilon}^h & ( u^h,u(t,x);t,x)}
\newcommand \Ruuh	{R_{\delta,\epsilon}^h(u,u^h)}
\newcommand \Suuh 	{S_{\delta,\epsilon}^h(u,u^h)}

\begin{document}
\title{Hyperbolic conservation laws on manifolds.
\\
Error estimate for finite volume schemes}

\author{
Philippe G. LeFloch$^1$, Wladimir Neves$^2$, 
\\
and Baver
Okutmustur$^1$}

\date{}

\maketitle

\footnotetext[1]{ Laboratoire Jacques-Louis Lions, Centre
National de la Recherche Scientifique, Universit\'e de Paris 6, 4
Place Jussieu,  75252 Paris, France. E-mail: {\sl
LeFloch@ann.jussieu.fr, Okutmustur@ann.jussieu.fr.} }

\footnotetext[2]{ Instituto de Matem\'atica, Universidade Federal
do Rio de Janeiro, C.P. 68530, Cidade Universit\'aria 21945-970,
Rio de Janeiro, Brazil. E-mail: {\sl Wladimir@im.ufrj.br.}
\newline
\textit{To appear in:} Acta Mathematica Sinica.   
\textit{AMS Subject Classification.} {Primary: 35L65.
Secondary: 76L05, 76N.}
\textit{Key words and phrases.} Hyperbolic conservation law, entropy solution,
finite volume scheme, error estimate, discrete entropy inequality, convergence rate.}

\begin{abstract}
Following Ben-Artzi and LeFloch, we consider nonlinear hyperbolic conservation laws 
posed on a Riemannian manifold, and we establish an $L^1$-error estimate for a class of finite
volume schemes allowing for the approximation of entropy solutions to the initial value problem. 
The error in the $L^1$ norm is of order $h^{1/4}$ at most, where $h$ represents the maximal diameter 
of elements in the family of geodesic triangulations. The proof relies on a suitable generalization 
of Cockburn, Coquel, and LeFloch's theory which was originally developed in the Euclidian setting. 
We extent the arguments to curved manifolds, by taking into account the effects to the geometry 
and overcoming several new technical difficulties.
\end{abstract}

\maketitle
 
\section{Introduction and background} \label{IN}
\subsection{Purpose of this paper} 

The mathematical theory of hyperbolic
conservation laws posed on curved manifolds $M$ was initiated by Ben-Artzi
and LeFloch \cite{BL}, and developed together with collaborators \cite{ABL,ALO,BFL,LeFloch,LO,LO2}. For these equations, a suitable generalization of
Kruzkov's theory has now been 
established and provides the existence and
uniqueness of an entropy solution to the initial and boundary
value problem for a large class of hyperbolic conservation laws and manifolds. The convergence of the finite volume schemes with
monotone flux was also established for conservation posed on manifolds.

The purpose of the present paper is to show that the error estimate for finite volume methods, due
to Cockburn, Coquel, and LeFloch \cite{CCL2} in the Euclidian
setting carries over to curved manifolds. To this end, we will need to revisit
Kuznetzov's approximation theory \cite{KNN1,KNN2} and adapt
the technique developed in \cite{CCL2}. One technical difficulty
addressed here is the adaption of the standard ``doubling of variables'' technique to curved manifolds.
We recover that the rate of error in the $L^1$ norm is of order $h^{1/4}$, where $h$ is 
the maximal diameter of an element of the triangulation of the manifold,
as first discovered in \cite{CCL2}.   

Recall that the well-posedness theory for hyperbolic conservation laws posed on a compact manifold was established 
in \cite{BL}, while the convergence of monotone finite volume schemes was proved in \cite{ABL}. In both papers,
DiPerna's measure-valued solutions \cite{DP} were used and can be
viewed as a generalization of Kruzkov's theory \cite{K}. In
contrast, in the present paper we rely on Kuznetsov's theory,
which allows us to bypass DiPerna's notion of measure-valued
solutions. Indeed, our main result in this paper provides both an
error estimate in the $L^1$ norm and, as a corollary, the actual convergence of the scheme to the
entropy solution; this result can be used to establish the existence of this entropy solution.

For another approach to conservation laws on manifolds we refer to Panov \cite{Panov} 
and for high-order numerical methods to Rossmanith, Bale, and LeVeque \cite{RBL} and the references therein. 
Concerning the Euclidian case $M=\RR^n$ we want to mention that the work by Cockburn, Coquel, and LeFloch \cite{CCL1,CCL2} 
(submitted and distributed in 1990 and 1991, respectively)  
was followed by important developments and applications by Kr\"oner \cite{Kroener} and Eymard, Gallouet, and Herbin \cite{EGH}
to various hyperbolic problems including also elliptic equations.  
In \cite{CCL1}, the technique of convergence using measure-valued solutions 
goes back to pioneering works by Szepessy \cite{Szepessy,Szepessy2} 
and Coquel and LeFloch \cite{CL1,CL2,CL3}. 
Concerning the error estimates we also refer to Lucier \cite{Lucier,Lucier2}, as well
as to Bouchut and Perthame \cite{BP} 
where the Kuznetsov theory is revisited.

An outline of this paper follows. In the rest of the present section we present some background on conservation 
laws on manifolds and briefly 
recall the corresponding well-posedness theory. Then in Section~2 we present the class of schemes under consideration 
together with the error estimate. Sections~3 and 4 contain estimates for various terms arising in the decomposition 
of the $L^1$ distance between the exact and the approximate solutions. 
The proof of the main theorem is given at the beginning of Section~4.

\subsection{Conservation laws on a manifold} \label{CL}

Let $(M,g)$ be a connected, compact, $n$-dimensional, smooth manifold endowed with a smooth metric $g$, 
that is, a smooth and non-degenerate $2$-covariant tensor field: 
for each $x \in M$, $g_x$ is a scalar product on the tangent space $T_xM$ at $x$.    
For any tangent vectors $X,Y \in T_xM$, we use the notation $g_x(X,Y) = \langle X,Y \rangle_g$ and 
$|X|_g := \langle X,X \rangle_g^{1/2}$.
We denote by $d_g$ the associated distance function 
and by $dv_g=dv_M$ the volume measure determined by the metric. 
Moreover, we denote by $\nabla_g$ the Levi-Civita connection associated with $g$.  
The divergence operator $\dive$ of a vector field is defined intrinsically as the trace of its covariant derivative.
It follows from the Gauss-Green formula that 
for every smooth vector field and any smooth open subset $S \subset M$
$$
\int_S \dive f\; dv_M = \int_{\del S} \langle f,n \rangle_g \, dv_{\del S},
$$
where $\del S$ is the boundary of $S$, $n$ is the outward unit normal along 
$\del S$, and  $dv_{\del S}$ is the induced measure on $\del S$.

Consider local coordinates $(x^i)$ together with the associated basis of tangent vectors
$\{\e_i\}= \{\del_i \}$ and covectors $\{\e^i\}$. The
differential of a function $u:M \to \RR$ is the differential form
$du= (du)_i  \e^i = {\del u \over \del x^i} e^i$, 
where the summation convention over repeated indices is used. The vector field $\nabla_g u$
associated with $du$ is given by
$\nabla_g u= (\nabla_g u)^i \; \e_i = g^{ij} \, (du)_j e_i$, 
where $(g^{ij})$ is the inverse of the matrix  $(g_{ij}) = \big( \langle \e_i,\e_j\rangle_g\big)$. 
The covariant derivative of a vector field $X$ is a $(1,1)$-tensor field whose coordinates are denoted by 
$(\nabla_g X)^j_k$.  
The following formula for the divergence of a smooth vector field will be useful: 
$$
\begin{aligned}
  \dive\big(f(u,x)\big)&= du(\partial_u f(u,x)) + \big(\dive f\big)(u,x)
  \\
  &= \partial_u f^i \: {\del u \over \del x^i} +
  \frac{1}{\sqrt{|g|}} \partial_i (\sqrt{|g|} \; f^i).
\end{aligned}
$$

We will use the following standard notation for function spaces defined on $M$. 
For $p \in [1,\infty]$ the usual norm of a function $h$ in the Lebesgue space $L^p(M;g)$ is denoted 
by $\| h \|_{L^p(M;g)}$
and, when $p=\infty$, we also write $\| h \|_\infty$. For any $f\in L_\loc^1(M; g)$ and any open subset $N \subset M$ we use the notation
$$ 
\dashint_N f(y)\; dv_g (y) := |N|_g^{-1}\int_N f(y) \; dv_g(y),
\qquad 
|N|_g := \int_N dv_g. 
$$

\subsection{Well-posedness theory}

We are interested in the following initial-value problem posed on the manifold $(M,g)$
\begin{eqnarray}
\label{CON1}     
u_t + \dive\big(f(u,\cdot)\big) &=& 0 \, \, \, \qquad \quad \text{ on } \RR_+ \times M, 
\\
\label{CON2}    
u(0,\cdot) & =&  u_0 \qquad \quad \text{ on } M,
\end{eqnarray}
where $u: \RR_+ \times M \to \RR$ is the unknown and  
the flux $f= f_x(u) = f(\overline u,x)$ is a smooth vector field which is 
defined for all $x \in M$ and also depends smoothly upon the 
real parameter $\overline u$.  
The initial data in \eqref{CON2} is assumed to be measurable and bounded, i.e. $u_0 \in L^\infty(M)$.
Moreover, $f$ satisfies the following growth condition 
\begin{eqnarray} 
\label{KC3} 
\max_{x \in M} |\big(\dive f\big)(u,x)| \leq C + C' \, |u|, \qquad u \in \RR 
\end{eqnarray}
for some constants $C, C'>0$.

\begin{definition}  A pair $(U,F$ is called an entropy pair if $U:\RR \to \RR$ is a Lipschitz continuous
function and $F=F(u,x)$ is a vector
field such that, for almost all $\bar{u} \in \RR$ and all $x \in M$, 
$$
    \del_u F(\bar{u},x) = \del_u U(\bar{u})
    \del_u f(\bar{u},x).
$$
If $U$ is also convex, then $(U,F)$ is called a convex entropy
pair.
\end{definition}

The most
important example of convex entropy pairs is the family of Kruzkov's
entropies, defined for $u, c \in \RR$ by 
\be \label{KRUZ}
    \begin{aligned}
     \big( U(u,c),F_x(u,c)\big) &:= \Big( |u-c|, \:
     \sgn(u-c) (f_x(u)- f_x(c)) \Big)\\
     &\;= \Big( (u \vee c - u \wedge c), \:
     f_x(u\vee c)- f_x(u\wedge c)\Big),
    \end{aligned}
\ee 
where $u\vee c = \max \{u,c\}$,  $u \wedge c = \min \{u,c\}$.

\begin{definition}
\label{ENT.SOL}
A function $u \in L^{\infty}(\RR_+ \times M)$ is called an entropy solution
to the initial value problem \eqref{CON1}-\eqref{CON2}, if for
every entropy pair $(U,F)$ and all smooth functions $\phi =
\phi(t,x) \geq 0$ compactly supported in $[0,\infty) \times M$,
\be 
\begin{aligned}
      & \iint_{\RR_+ \times M} \Big( U(u) \, \phi_t +
      \langle F_x (u), \nabla_g \phi \rangle_g \Big) \; dv_gdt\\
      & + \iint_{\RR_+ \times M} G_x(u)\phi(t,x) \, dv_g dt + \int_{M} U(u_0(x)) \phi(0) \; dv_g \geq 0,
\end{aligned}
\ee
where $G_x(u):= (\dive F_x)(u) - \del_u U(u) (\dive f_x)(u)$.
\end{definition}

For instance, with Kruzkov's entropies the above definition becomes (for all $c \in \RR$)
$$
 \begin{aligned}
 &\iint_{\RR_+ \times M} \Big(U (u,c) \, \phi_t + \langle F_x(u,c),\nabla_g \phi \rangle_g \Big)\, dv_g dt
\\
   & -\iint _{\RR_+ \times M} \sgn(u-c) \, (\dive f)(c,x)\; \phi \, dv_g dt
   +  \int_M |u_0 - c| \, \phi(0) \, dv_g \geq 0.
 \end{aligned}
$$
The well-posed theory for the initial value problem \eqref{CON1}-\eqref{CON2} 
was established in Ben-Artzi and LeFloch \cite{BL}. 

In the present paper, we are interested in the discretization of the problem \eqref{CON1}-\eqref{CON2}
in the case that the initial data is bounded and has finite total variation
\be 
\label{IDR}
u_0 \in L^\infty(M) \cap BV(M; g).
\ee

In particular, it is established in \cite{BL} that 
in the case of bounded initial data, the following variant of the maximum principle is established: 
$$ 
   \|u(t)\|_{L^\infty(M)} \leq C_0(T,g) + C_0'(T,g) \:
   \|u(s)\|_{L^\infty(M)}, \quad \quad 0\leq s \leq t \leq T,
$$
where the constants $C_0,C_0'>0$ depend on $T$ and the metric $g$.

Recall the definition of the total variation of a function $w: M \to \RR$  
$$
  \TV_g(w):= \sup_{\|\phi\|_\infty \leq 1} \int_M w \: \dive\!\phi \: dv_g,
$$
where $\phi$ describes all $C^1$ vector fields with compact support. We denote by
$$
  BV(M; g) = \{u \in L^1(M; g) \, / \, \TV_g(u) < \infty \},
$$
the space of all functions with finite total variation on $M$. It is well-known that 
(provided $g$ is sufficiently smooth) the 
imbedding $BV(M; g) \subset L^1(M; g)$ is compact.

In fact, an important property of entropy solutions to \eqref{CON1}-\eqref{CON2} is the following one: 
$u$ has finite total variation for all times $t \geq 0$ if \eqref{IDR} holds 
and, moreover, 
$$
\label{TVE}
  \TV_g(u(t)) \leq C_1(T,g) + C_1'(T,g) \: \TV_g(u(s)), \quad \quad 0 \leq s \leq t \leq T,
$$
where the constants $C_1, C_1' >0$ depend on $T$ and $g$; see \cite{BL} for details.
Of course, this implies a control of the flux
of the equation
$$
\sup_{t \geq 0} \int_M  \Big|\dive \big(f(u(t,\cdot),\cdot)\big)\Big| dv_g
 \leq C \, \TV_g(u_0). 
$$   
However, as noted in \cite{ABL}, this inequality can be derived more directly from the conservation 
laws and one checks that the constant $C$ is independent of both $T$ and $g$ and only depend on the largest 
wave speed arising in the problem.


\section{ Statement of the main result}
\subsection{Family of geodesic triangulations}

For $\tau>0$, we consider the uniform mesh $t_n:= n \, \tau$ ($n=0,1,2, \ldots$) on the half-line $\RR_+$.
For $h>0$ we denote by $\Tcal^h$ a triangulation of the given manifold $M$ 
which is made of 
non-overlapping and non-empty curved polyhedra $K \subset M$,
whose vertices in $\del K$ are joined by geodesic faces. We assume that, if
two distinct elements $K_1, K_2 \in \Tcal^h$ have a non-empty
intersection, say $I$, then either $I$ is a geodesic face of both
$K_1$, $K_2$ or $\Hcal^{n-1}(I) = 0$, where $\Hcal^{n-1}$ denotes the $(n-1)$-dimensional Hausdorff measure. 

The boundary $\dK$ of $K$ consists of the set of all faces $e$ of $K$. We denote by $K_e$ the unique element distinct
from $K$ sharing the face $e$ with $K$. The outward unit normal to
an element $K$ at some point $x \in e$ is denoted by $\nek(x) \in T_xM$. Finally, $|K|$ and $|e|$
represent the $n$- and $(n-1)$-dimensional Hausdorff measures of $K$ and $e$, respectively. 
We set
$$
  p_K := \sum_{e\in\dK}|e|
$$
and for each $K\in \TT^h$ the diameter $h_K$ of $K$ is
$$ 
h_K:= \sup_{x,y\in K} d_g(x,y).
$$

We set 
$$
h := \, \sup \{ h_K: K\in \TT^h\},
$$
 which is assumed to tend to zero along a sequence of geodesic triangulations.
We also assume that there exist constants $\gamma_1, \gamma_2 > 0$ such that
\be 
\label{IN1}
   \gamma_1^{-1}h \leq \tau \leq \gamma_1 h
\ee and \be\label{IN2}
    \gamma_2^{-1} |K| \leq h_K \; p_K   \leq \gamma_2 |K|
\ee
for all $K\in \TT^h$. This condition implies that (as $h \to 0$) 
$$
 \tau \to 0, \qquad \ h^2\tau^{-1} \to 0.
$$
Finally, we set $T = \tau \; n_T$ for every integer $n_T$.


\subsection{Numerical flux-functions}

As in the Euclidean case, the finite volume method can be introduced by formally 
averaging the conservation law \eqref{CON1} over an element $K\in \TT^h$, applying the Gauss-Green formula, 
and finally discretizing the time derivative with a two-point scheme. First,
we define a right-continuous, piecewise
constant function: for $n=0,1,\ldots$, 
\be 
\label{PCF}
   u^h(t,x)= u_K^n \qquad (t,x) \in [ t_n, t_{n+1}) \times M,
\ee
where
$$ 
u_K^n := \dashint_{K} u(t_n,x) \, dv_g (x),
$$ and
\be  \label{PCF2}
 u_K^0 := \dashint_{K}
u_0(x)dv_g(x).
\ee
Then, in view of \eqref{CON1} we write 
$$
\begin{aligned}
    0 &= \frac {d}{dt} \dashint_{K} u(t,x) \; dv_g(x) + \dashint_{K}
    \dive f(u(t,x),x) \; dv_g(x)
\\[8pt]
& \approx \frac {u_K^{n+1}-u_K^n}{\tau} +  \frac {1}{|K|}
\sum_{e\in \dK} \int_e \langle f(u(t,y)),n_{e,K}(y)\rangle_g \;
d\Gamma_g(y).
\end{aligned}
$$
We introduce flux-functions $f_{e,K} :\RR \times \RR \to \RR$ 
and write 
$$
\dashint_{e} \langle f(w(u_K^n;u^n_{K_e}),y),n_{e,K}(y)\rangle_g \; d\Gamma_g(y)
 \approx  f_{e,K}(u_K^n,u_{K_e}^n).  
$$

The discrete flux are assumed to satisfy the following properties:
\begin{itemize}
\item{\emph{Consistency property :} for $u\in \RR,$
\be
    \label{FVM.1}
    \fek(u,u) = \dashint_{e} \langle f(u,y), \nek(y)\rangle_g\,
    d\Gamma_g(y).
    \ee
    }
\item{\emph{Conservation property :} for $u,v \in \RR,$
        \be
    \label{FVM.2}
    \fek(u,v)+ \feke(v,u)=0.
\ee }

\item{\emph{Monotonicity property:}
 \be
    \label{FVM.3}
    \frac{\del}{\del u} \fek \ge 0, \qquad \frac{\del}{\del v} \fek \le
    0.\\[5pt]
    \ee
}
\end{itemize}

Then, we formulate the finite volume approximation as follows: 
\be 
\label{FVM} 
u_K^{n+1}:= u_K^n - \frac{\tau}{|K|} \sum _{e\in
\dK} |e| \; f_{e,K}(u_K^n, u_{K_e}^n) \quad \qquad (n=0,1,\ldots).
\ee
For the sake of stability of the numerical method, we impose a CFL stability condition:  
\be 
\label{CFL}
\tau \,  \sup_{K\in \TT^h} \frac{p_K}{|K|} \, \Lip(f) \leq 1,
\ee
where $\Lip(f)$ is the Lipschitz constant of $f$.

\subsection{Main theorem}

The main result of the present paper is as follows. 

\begin{theorem}[Error estimate for the finite volume scheme on manifolds]
\label{MT} 
Let $u: \RR_+\times M \to \RR$ be the
entropy solution associated with the initial value problem \eqref{CON1}-\eqref{CON2}
for an initial data $u_0 \in L^\infty(M) \cap BV(M; g)$. 
Let $u^h$ be the approximate solution defined by \eqref{PCF} and
\eqref{FVM}. Then, for each $T>0$ there exist constants 
$$
\aligned
  &\ct_0 = \ct_0(T,g,\|u_0\|_{L^\infty}), \quad
  \ct_1= \ct_1(T,g,\TV_g(u_0)), 
  \\
  & 
  \ct_2= \ct_2(T,g,\|u_0\|_{L^2(M;g)})
\endaligned
$$
such that for all $t \in [0,T]$  
$$
\aligned
& \|u^h(t) -u(t)\|_{L^1(M; g)} 
\\
& \leq  \Big( \ct_0 \, |M|_g + \ct_1 \Big) \: h
   + \Big(\ct_0 \, |M|_g^{1/2} + \big(\ct_0 \, \ct_1\big)^{1/2}
  \Big) \: |M|_g^{1/2} \; h^{1/2}
  \\
  & \quad + \Big( \big(\ct_0 \, \ct_2 \big)^{1/2}  |M|_g^{1/2} + \big(\ct_1 \, \ct_2\big)^{1/2}
  \Big) \: |M|_g^{1/4} \; h^{1/4}.
  \endaligned
$$
\end{theorem}

The rest of the present paper will be devoted to the proof of this theorem, which will follow
from a suitable generalization of the arguments introduced earlier 
in Cockburn, Coquel, and LeFloch \cite{CCL2}.

\begin{remark}
1. It is sufficient to establish Theorem~\ref{MT} for smoother initial data. 
When $u_0$ is measurable and bounded on $M$ one can then show the existence of weak solutions 
to the initial value problem \eqref{CON1}-\eqref{CON2} as follows. 
Let $u_0^h \in L^\infty(M) \cap BV(M)$ be such that
$$
  \lim_{h \to 0} u_0^h = u_0 \quad \quad \text{in $L^1(M; g)$}.
$$
Solving the corresponding problem \eqref{CON1}-\eqref{CON2} for the regularized initial data, we 
deduce from Theorem~\ref{MT} that the
approximation solutions $\{u^h\}_{h>0}$ form a Cauchy sequence in
$L^1$. Moreover, in view of the (discrete) maximum principle established later in this paper
and for every $T>0$, these solutions are uniformly bounded in $L^\infty((0,T) \times M)$.
Consequently, there is a function $u \in L^\infty$, such that
$$
  \lim_{h \to 0} u^h = u \quad \quad \text{in the $L^1$ norm.}
$$
Finally, we note that Definition~\ref{ENT.SOL} is stable in the $L^1$ norm.

2. An immediate consequence of the $L^1$-contraction property is
an estimate of the modulus of continuity in time, that is
$$
   \|u^h(t) - u^h(s)\|_{L^1(M; g)} \leq C \; h + C' \;  |t-s|,
$$
where the constants $C>0$, $C' \geq 1$ may depend
on $T$ as well as the metric $g$.
\end{remark}


\subsection{Discrete entropy inequalities}

We will rely on a discrete version of the entropy inequality 
formulated by expressing $u_K^{n+1}$ as a convex combination of essentially one-dimensional schemes. For
each $K \in \TT^h$, $ e\in \dK$, we define
\be 
\label{CC1} \tunnkee := \unk - \frac{\tau p_K}{|K|} \big(
\fek(\unk, \unke)- \fek(\unk, \unk) \big) 
\ee
and set
\be \label{CC2}
   \unnkee := \tunnkee -  \frac{\tau p_K}{|K|} f_K^n,
\ee
where
$$
    f_K^n:= \frac{1}{p_K} \sum _{e\in \del
    K}|e|f_{e,K} (u_K^n,u_K^n).
$$
Therefore, in agreement  with \eqref{FVM} we find 
\be 
\label{CC4}
    u_K^{n+1} = \frac{1}{p_K} \sum _{e\in \dK} |e| \;
    u_{K,e}^{n+1}.
\ee

\begin{remark}
The error estimate will be derived from the family of Kruzkov's
entropies. Recall that any smooth entropy $\eta(u)$
can be recovered by the family of Kruzkov's entropies, that is
$$
  \eta(u) = \frac{1}{2} \int_\RR \partial^2_u \eta(\xi) \: U(u,\xi) \:
  d\xi.
$$
The result follows for any entropy by a standard regularization
argument. Moreover, if $(\eta,q)$ is any convex entropy pair, then
$$
  q_x(u) = \frac{1}{2} \int_\RR \partial^2_u \eta(\xi) \: F_x(u,\xi) \:
  d\xi.
$$
\end{remark}

\

Define the numerical family of Kruzkov's entropy-flux as
\be \label{KRUZN}
  F_{e,K}(u,v,c):= f_{e,K}(u\vee c,v \vee c)- f_{e,K}(u\wedge c,v\wedge
  c).
\ee
Given any convex entropy pair $(U,F)$, define
the numerical entropy-flux $F_{e,K}(u,v)$ associated to $F$
 by
$$
  F_{e,K}(u,v) := \frac{1}{2} \int_\RR \partial^2_u U(\xi) \: F_{e,K}(u,v,\xi) \:
  d\xi.
$$
Hence, from the condition of the discrete flux we see that $F_{e,K}(u,v)$ satisfies
\begin{itemize}
\item{For $u\in \RR$,
$$
    F_{e,K}(u,u) = \dashint_{e} \langle F_y(u), \nek(y)\rangle_g \,
    d\Gamma_g(y).
$$
    }
\item{For $u,v \in \RR$,
$$
    F_{e,K}(u,v)+ F_{e,K_e}(v,u)=0.
$$
    }
\end{itemize}
These properties are inherited from the corresponding properties for the
numerical family of Kruzkov's entropy-flux.

We are in a position to derive the discrete entropy inequalities. 
For each $K \in \TT^h$, $
e\in \dK$ and $u,v \in \RR$, we define
$$
   H_{e,K}(u,v) := u - \varpi_K \big( f_{e,K}(u,v) - f_{e,K}(u,u) \big),
$$
where
$$
  \varpi_K := \frac{\tau p_K}{|K|}.
$$
Hence from \eqref{CC1}, $H_{e,K}(\unk,\unke)=\tunnkee $ and by
definition of $H_{e,K}$, we have
$$
  \partial_u H_{e,K}(u,v) \geq 0, \quad \partial_v H_{e,K}(u,v) \geq
  0.
$$
The last inequality is an immediate consequence of the
monotonicity of $f_{e,K}(u,v)$. The former follows from this
property and the CFL condition. Moreover, we observe that
\be \label{HH2}
    \begin{aligned}
    &H_{e,K}(u \vee \lambda,v \vee \lambda) - H_{e,K}(u \wedge \lambda,v\wedge\lambda)
    \\[5pt]
    &=\Big( u\vee\lambda - \varpi_K \, (f_{e,K}(u\vee\lambda,v\vee\lambda)
    - f_{e,K}(u\vee\lambda,u\vee\lambda)) \Big) \\[5pt]
    & \quad - \Big( u\wedge\lambda -
    \varpi_K \, ( f_{e,K}(u\wedge\lambda,v\wedge\lambda)
    - f_{e,K}(u\wedge\lambda,u\wedge\lambda)) \Big)\\[5pt]
    &=(u\vee\lambda - u\wedge\lambda)
     -\varpi_K \, (f_{e,K}(u\vee\lambda,v\vee\lambda)
      - f_{e,K}(u\wedge\lambda,v\wedge\lambda))  \\[5pt]
    & \quad +\varpi_K \, ( f_{e,K}(u\vee\lambda,u\vee\lambda)
      - f_{e,K}(u\wedge\lambda,u\wedge\lambda)) \\[5pt]
    &= U(u,\lambda) - \varpi_K \, \big(F_{e,K}(u,v,\lambda) -
    F_{e,K}(u,u,\lambda)\big),
   \end{aligned}
\ee where we have used \eqref{KRUZN}. Now, since
$H_{e,K}(u,v)$ is an increasing function in both variables, we
have
$$
  \begin{array}{l}
     H_{e,K}(u,v) \vee H_{e,K}(\lambda,\lambda) \leq
     H_{e,K}(u\vee\lambda,v\vee\lambda), \\[5pt]
     H_{e,K}(u,v) \wedge H_{e,K}(\lambda,\lambda) \geq
     H_{e,K}(u\wedge\lambda,v\wedge\lambda), 
   \end{array}
$$
hence,
\be \label{HH3}
    \begin{aligned}
    &H_{e,K}(u \vee \lambda,v \vee \lambda) - H_{e,K}(u \wedge \lambda,v\wedge\lambda)
    \\[5pt]
    &\geq \big( H_{e,K}(u,v) \vee H_{e,K}(\lambda,\lambda) \big)
    - \big( H_{e,K}(u,v) \wedge
    H_{e,K}(\lambda,\lambda)\big)\\[5pt]
    &= U(H_{e,K}(u,v),\lambda).
   \end{aligned}
\ee
Consequently, from \eqref{HH2},\eqref{HH3} taking $u=\unk$ and
$v=\unke$, we obtain
$$
  U(\tunnkee) - U(\unk) + \frac{\tau p_K}{|K|} \Big( \Fek(\unk, \unke)
    - \Fek(\unk, \unk) \Big) \le 0.
$$
\\
Therefore, we have proved: 

\begin{lemma}[Entropy inequalities for the finite volume scheme] 
\label{Lem1}
 Let $(U,F)$ be a convex entropy pair. Then,
 there exists a family of Lipschitz functions $F_{e,K}: \RR^2 \to \RR$, called numerical
 entropy-flux associated to $F$, satisfying the
 following conditions:
\begin{itemize}
\item{\emph{Consistency property :} for $u\in \RR$,
    \be
    \label{FVM.4}
    F_{e,K}(u,u) = \dashint_{e} \langle F_y(u), \nek(y)\rangle_g \,
    d\Gamma_g(y).
    \ee
    }
\item{\emph{Conservation property :} for $u,v \in \RR$,
    \be
    \label{FVM.5}
    F_{e,K}(u,v)+ F_{e,K_e}(v,u)=0.
    \ee
    }
    \item{\emph{ Discrete entropy inequality:}
    \be
    \label{FVM.6}
    U(\tunnkee) - U(\unk) + \frac{\tau p_K}{|K|} \Big( \Fek(\unk, \unke)
    - \Fek(\unk, \unk) \Big) \le 0.
    \ee
}
\end{itemize}
\end{lemma}

From \eqref{CC2} and \eqref{FVM.6}, we can write the discrete
entropy inequality in terms of $\unnkee$ and $\unk$, that is
\be \label{FVM.7}
       U(\unnkee) - U(\unk) +
       \frac{\tau p_K}{|K|} \Big( \Fek(\unk, \unke)
       - \Fek(\unk, \unk) \Big) \le \dnnkee,
\ee where $\dnnkee = U(\unnkee) - U(\tunnkee)$.

\

To end this section we also recall the discrete maximum principle
 established in Amorim, Ben-Artzi and LeFloch \cite{ABL}: for $n=0,1,\ldots,n_T$,
$$
   \max_{K \in \TT^h} |u^n_K| \leq \big(\tilde{C_0}(T) +
   \max_{K\in \TT^h} |u^0_K| \big) \: \tilde{C_0'}(T),
$$
for some constants $\tilde{C_0}(T), \tilde{C_0'}(T) >0$.

\section{Derivation of the error estimate}

\subsection{Fundamental inequality}

From now on it will be convenient to use the notation $Q_T := [0,T]\times M$.  
In this section we derive a basic
approximation inequality on the manifold $M$, that is, we derive a
generalization of the Kuznetsov's approximation inequality for the $L^1$ distance 
$$
   \|u^h(T) - u(T)\|_{L^1(M; g)}.
$$
Before proceeding we need introduce some special test-functions and make some preliminary
observations. 

Let $\varphi:\RR \to \RR$ be any $C^\infty$ function such that $\text{supp} \varphi \subset [0,1]$, 
$\varphi \geq 0$, and $\int \varphi = 1$.  
For each $t'
\in \RR$, $x' \in M$ be fixed, each $\delta, \epsilon > 0$ and
all $t \in \RR$, $x \in M$, we define
\be 
\label{STF}
    \rho_{\delta}(t;t'):= \frac{1}{\delta} \: \varphi\Big(
    \frac{|t-t'|^2}{\delta^2}\Big),
\hspace{20pt}
    \psi_{\epsilon}(x;x'):= \frac{1}{\epsilon^n} \: \varphi\Big(
    \frac{(d_g(x,x'))^2}{\epsilon^2}\Big),
\ee
where we use the Riemannian distance. Observe that
$\rho_{\delta}(t;t')=\rho_{\delta}(t';t)$, 
$\psi_{\epsilon}(x;x')=\psi_{\epsilon}(x';x)$, and  
$$
   \int_\RR \rho_{\delta}(t;t')\, dt = 1, \qquad
    \int_M \psi_{\epsilon}(x;x') \,  dv_g(x) = 1.
$$

Clearly, $\psi_{\epsilon}$ is a Lipschitz function on $M$ with
compact support contained in the geodesic ball of radius
$\epsilon$, hence $\psi_{\epsilon} \in W^{1,\infty}(M)$ and
by  Rademacher's theorem~\cite{EG} it is differentiable
almost everywhere. Moreover, there exists a constant $C>0$, such
that for $\mathcal{L}^1-$a.e. $t \in \RR$ and $\mathcal{H}^n-$a.e.
$x \in M$
\be \label{ETF}
\begin{aligned}
   &|\rho_{\delta}(t;t')| \leq \frac{C}{\delta},
   \qquad
   |\partial_t \rho_{\delta}(t;t')| \leq \frac{C}{\delta^2},
\\[5pt]
   &|\psi_{\epsilon}(x;x')| \leq \frac{C}{\epsilon^n},
   \qquad
   |\nabla_g \psi_{\epsilon}(x;x')| \leq \frac{C}{\epsilon^{n+1}}.
\end{aligned}
\ee
Let $v:M \to \RR$ be a locally integrable function on $M$. We
may assume only on some $N \subset M$ and to be extended by zero
on $M \backslash N$. Therefore, there exists a sequence of
smooth functions $\{v^m\}$ defined on $M$, such that
$$
  \lim_{m \to \infty} v^m = v \qquad \text{on $L^1(M; g)$}.
$$

As in Kruzkov \cite{K}, in the case of manifolds we can establish the 
following approximation result. 

\begin{lemma} \label{KL}
Given a bounded and measurable function $v: M \to \RR$ and
$\epsilon>0$, the function 
$$
  V_\epsilon := \int_M \int_M \psi_\epsilon(x;x') \: |v(x) -
  v(x')| \: dv_g(x) \, dv_g(x')
$$
satisfies 
$$
  \lim_{\epsilon \to 0} V_\epsilon = 0.
$$
\end{lemma}

\begin{proof}
First, consider the case that $v$ is smooth on $M$. Let
$x$, $x'$ be two points on $M$ and, $\gamma:[0,1] \to M$ be a
minimizing geodesic with $\gamma(0)=x$ and $\gamma(1)=x'$.
Therefore, for some $\xi \in (0,1)$ one can write 
$$
   \begin{aligned}
   |v(x) - v(x')| &= |v(\gamma(1)) - v(\gamma(0))|
    = \big|\frac{d}{dt}(v \circ \gamma)(\xi)\big|
   \\
   &= |\langle du(\gamma(\xi)),\big(\frac{d\gamma}{dt}\big)(\xi)\rangle|
    \leq |\nabla_g v(\gamma(\xi))| \: \big|\frac{d\gamma}{dt}(\xi)\big|
   \\
   & \leq \|\nabla_g v\|_\infty \: d_g(x,x').
   \end{aligned}
$$

Then, applying the inequality above and using \eqref{ETF}, we
obtain
$$
   \begin{aligned}
   V_\epsilon &= \int_M \int_M \psi_\epsilon(x;x') \: |v(x) -
  v(x')| \: dv_g(x') \, dv_g(x)
  \\
  & \leq \int_M \frac{C}{\epsilon^n} \: \|\nabla_g v\|_\infty V_g\big(B_x(\epsilon)\big) \: \epsilon \, dv_g(x)
  \\
  & \leq C \: |M| \: \|\nabla_g v\|_\infty \: \epsilon,
   \end{aligned}
$$
where $V_g\big(B_x(\epsilon)\big)$ denotes the volume of the
geodesic ball of center $x$ and radius $\epsilon$, and the
constant $C>0$ does not depend on $\epsilon$.
 
Finally, suppose that $v$ is measurable and bounded on $M$.
Since $M$ is compact, $v$ is integrable on $M$ and there
exists a sequence of smooth functions $\{v^m\}$ defined on $M$
converging to $v$ in $L^1(M; g)$. We then conclude with a routine approximation argument.
\end{proof}

For each $p,p' \in Q_T$, $p:=(t,x)$, $p':=(t',x')$, we consider
the special test function $\phi$ defined by
$$
     \phi(p\,;p'):= \rho_{\delta}(t;t') \: \psi_{\epsilon} (x;x').
$$
Therefore, as $\delta, \epsilon \to 0$,
 the support of $\phi$ is concentrated on the set
$\{p =  p'\}$. 
For convenience,
we introduce the following piecewise approximation of the
functions $\rho_{\delta}$ and $\psi_{\epsilon}$. For
$n=0,1,\ldots$, we define $\tilde{\rho}(t;t')$,
$\tilde{\rho}'(t;t')$ as
\be \label{PAF1}
    \begin{aligned}
    & \tilde{\rho}_{\delta} (t;t') := \rho_{\delta}(t_{n+1};t') \qquad
    ( t \in [t_n, t_{n+1})),
\\[5pt]
   &\tilde{\rho}'_{\delta} (t;t') := \rho_{\delta}(t;t'_{n+1}) \qquad
    ( t' \in [t'_n, t'_{n+1})),
\end{aligned}
\ee
and for all $K \in \TT^h$, we define $\tilde{\psi}_{\epsilon}
(x;x')$, $\tilde{\psi}_{\epsilon}'(x;x')$ by the averages of
${\psi}_{\epsilon} (x;x')$ along each interface, that is
\be \label{PAF2}
\begin{aligned}
    &\tilde{\psi}_{\epsilon} (x;x') :=  \dashint_{\dK}
    \psi_{\epsilon}(y;x')\, d\Gamma_g(y) \quad ( x \in K, x' \in M ),
    \\[5pt]
    &\tilde{\psi}'_{\epsilon} (x;x') := \dashint_{\dK}
    \psi_{\epsilon}(x;y') \,d\Gamma_g(y') \quad ( x \in M, x' \in
    K).
\end{aligned}
\ee

The following estimate will be useful.

\begin{lemma} 
\label{KL1}
The functions  ${\psi}_{\epsilon}$, $\tilde{\psi}_{\epsilon}$ be defined by \eqref{STF}, and
 \eqref{PAF2}, respectively, satisfy the estimate 
$$
\sup_{x \in M}   \int_M |\tilde{\psi}_{\epsilon} (x;x') - {\psi}_{\epsilon}
  (x;x')| \: dv_g(x') \leq C \; \frac{h}{\epsilon},
$$
where $C>0$ does not depend on $\epsilon, h>0$.
\end{lemma}

\begin{proof} Given $K$, let $x',x$ be two points on $M$ and $K$,
respectively. Analogously to the proof of Lemma \ref{KL}, we
could write
$$
\begin{aligned}
  |\tilde{\psi}_{\epsilon} (x;x') - {\psi}_{\epsilon}
  (x;x')| &\leq \dashint_{\dK}
  |\psi_{\epsilon}(y;x')-\psi_{\epsilon}(x;x')| \: d\Gamma_g(y)
  \\
  & \leq |\nabla_g \psi_{\epsilon}(c;x')| \: h \leq C {h \over \epsilon^{n+1}}, 
\end{aligned}
$$
where
$$
  |\nabla_g \psi_{\epsilon}(c;x')| = \sup_{y \in \dK} |\nabla_g
  \psi_{\epsilon}(y;x')|.
$$

Now, we integrate the above inequality on $M$ and obtain 
$$
\begin{aligned}
  \int_M |\tilde{\psi}_{\epsilon} (x;x') - {\psi}_{\epsilon}
  (x;x')| \: dv_g(x') 
  & \leq \frac{C}{\epsilon^{n+1}} \: h \: V_g\big(B_c(\epsilon)\big)
  \leq C \: \frac{h}{\epsilon}.
\end{aligned}
$$
\end{proof}

Moreover, we define the corresponding approximations $\phi^h$,
$\phi^{h'}$, $\del_t^h \phi$, $\del_{t}^{h'}\phi$, $\nabla_g^h\phi$
and $\nabla_g^{h'}\phi$ of the exact value, time derivative and
covariant derivative of the function $\phi$, respectively:
$$
\aligned 
     & \phi^h(p \,; p'):=  \tilde{\rho}_{\delta}(t;t') \;
     {\psi}_{\epsilon}(x;x'),
\qquad 
\phi^{h'}(p \,; p'):= \tilde{\rho}'_{\delta}(t;t') \;
    {\psi}_{\epsilon}(x;x'),
\\
     & \del_t^h\phi(p \,; p'):= \del_t \rho_{\delta}(t;t') \;
     \tilde{\psi}_{\epsilon}(x;x'),
\qquad   \del_t^{h'}\phi(p \,; p'):= \del_{t'} \rho_{\delta}(t;t') \;
   \tilde{\psi}'_{\epsilon}(x;x'),
\endaligned
$$
and
$$
 \begin{aligned}
  &\nabla_g^h\phi(p \,; p'):= \tilde{\rho}_{\delta}(t;t') \;
     \nabla_g{\psi}_{\epsilon}(x;x'),
\\
&    \nabla_g^{h'}\phi(p \,; p'):=
     \tilde{\rho}'_{\delta}(t;t') \;
      \nabla_g'{\psi}_{\epsilon}(x;x').
\end{aligned}
$$
Analogously, for convenient we introduce a piecewise constant
approximation of the exact solution $u$, that is \be
\label{A.E.S.}
       \tilde{u}(t,x):= u(t_n,x), \qquad  (t \in [t_n,t_{n+1}), x \in M).
\ee

\begin{remark}
1. Note that $u$ represents the entropy solution to problem
\eqref{CON1}-\eqref{CON2}, while $u^h$ denotes the
piecewise-constant approximate solution \eqref{PCF} given by the
schema \eqref{FVM}. By definition we have $\tilde{u}^h = u^h$.

2. The zero-order approximations of the test function $\phi$, that
is, $\phi^h$ and $\phi^{h'}$ are due to the explicit dependence of
the flux function with the spacial variable.

3. One denotes by  $\partial_{t'}$, $\nabla_g'$ respectively the time
derivative and covariant derivative with respect to $t'$ and $x'$
variables.
\end{remark}

Next, let us define the approximate entropy dissipation form 
$$
   E_{\delta, \epsilon}^h(u,u^h):= \iint _{Q_T} \Theta_{\delta,
   \epsilon}^h(u,u^h(t',x');t',x')\, dv_g(x') dt',
$$
where
\be \label{MED}
\begin{aligned}
\Theta_{ \delta,\epsilon}^h & (u,u^h(t',x');t',x')
 \\ &= -\iint_{Q_T} \Big(|\tilde{u}(t,x) - u^h(t',x')| \,
 \del_t^h\phi
\\[5pt]
&  + \big(\sgn(u(t,x)- u^h(t',x') )\langle f(u(t,x),x) -
f(u^h(t',x'),x),\nabla_g^h\phi\rangle \big)  \Big) \, dv_g(x) dt
\\[5pt]
& +\iint _{Q_T} \sgn(u(t,x)-u^h(t',x')) \; (\dive f)(u^h(t',x'),x)
\; \phi^h \, dv_g(x) dt
\\[5pt]
& -  \int_M |u(0,x)- u^h(t',x')| \phi(0) \, dv_g(x)
 + \int_M |u(T,x)- u^h(t',x')| \phi(T) \, dv_g(x).
\end{aligned}
\ee
Here, the term $\Theta_{\delta,\epsilon}^h(u,u^h;t',x')$ is a
measure of the entropy dissipation associated with the entropy
solution $u$. Observe that $\tilde{u}$ defined by \eqref{A.E.S.}
appears in the first term of the right-hand side of \eqref{MED}.
This is due to the fact that the time derivative of $u^h$ needs
special treatment, as was observed in \cite{CCL1}.

Analogously, reversing the role of $u$ and $u^h$, we
define 
$$
   E_{\delta, \epsilon}^h(u^h,u) := \iint _{Q_T} \Theta_{\delta,
   \epsilon}^h(u^h,u(t,x);t,x)\, dv_g(x) dt,
$$
where
$$
\begin{aligned}
\Thuhu
 \\ := & -\iint_{Q_T} \Big(| u^h(t',x')- u(t,x) |\del  _t^ {h'} \phi
\\
&+ \big(\sgn(u^h(t',x') -u(t,x) ) \langle f (u^h(t', x'),x') - f (
u(t,x),x'),\nabla_g ^{h'} \rangle \big) \Big)\, dv_g(x') dt'
\\
& + \iint _{Q_T} \sgn(u^h(t',x') -u(t,x)) \; (\dive' f)(u(t,x),x')
\; \phi^{h'} \, dv_g(x') dt'
\\
& -  \int_M |u^h(0,x')- u(t,x)| \phi(0) \, dv_g(x') + \int_M
|u^h(T,x')- u(t,x)| \phi(T) \, dv_g(x').
\end{aligned}
$$

Observing that $\partial_{t'}\rho_\delta(t;t') =
-\partial_{t}\rho_\delta(t;t')$, $\nabla_g' \psi_\epsilon(x;x')=
-\nabla_g \psi_\epsilon(x;x')$ and adding the terms $E_{\delta,
\epsilon}(u,u^h)$ and $E_{\delta, \epsilon}(u^h,u)$, we get the
following decomposition:
\be
 \label{DECOMP1}
   \Euuh+ \Euhu = \Ruuh - \Suuh,
\ee
where
\be \label{DECOMP2}
\begin{aligned}
 \Ruuh := & \iint _{Q_T}\int_{M} |u^h(T,x')-u(t,x)| \; \phi(t,x;T,x')
dv_g(x') dv_g(x) dt
\\
& + \iint _{Q_T}\int_{M} |u^h(t',x')-u(T,x)| \; \phi(T,x;t',x')
dv_g(x) dv_g(x')dt'
\\
& - \iint _{Q_T}\int_{M} |u^h(0,x')-u(t,x)| \; \phi(t,x;0,x')
dv_g(x') dv_g(x) dt
\\
& - \iint _{Q_T}\int_{M} |u^h(t',x')-u(0,x)| \; \phi(0,x;t',x')
dv_g(x)dv_g(x') dt',
\end{aligned}
\ee and \be \label{DECOMP3}
\begin{aligned}
& \Suuh
\\
:= 
& \iint _{Q_T}\iint _{Q_T} \Big(|\tilde{u}(t,x)-u^h(t',x')|
 \; \tilde {\psi}_{\epsilon}(x;x') \\
 &\hspace{30pt} -|u(t,x)-u^h(t',x')| \; \tilde{\psi}'_{\epsilon}(x;x') \Big)
\, \del_t \rho_{\delta}(t;t') \, dv_g(x') dt' dv_g(x) dt
\\
& + \iint _{Q_T}\iint _{Q_T} \sgn(u(t,x)- u^h(t',x'))\\
& \hspace{30pt} \Big( \big\langle \big(f(u(t,x),x) -
f(u^h(t',x'),x)\big)\;
\tilde{\rho}_{\delta}(t;t') \\
& \hspace{40pt} - \big(f(u(t, x),x') -
f(u^h(t', x'),x')\big) \; \tilde{\rho}'_{\delta}(t;t'), \nabla_g \psi_{\epsilon}(x;x') \big\rangle_g \\
& + (\dive' f)(u(t,x),x') \phi^{h'} - (\dive f)(u^h(t',x'),x)
\phi^h \Big) \, dv_g(x') dt' dv_g(x) dt.
\end{aligned}
\ee
Passing to the limit as $\delta,\epsilon \to 0$, we expect that
$R_{\delta,\epsilon}^h(u,u^h)$ converges to
$$
\int_{M} |u^h(T,x)-u(T,x)| \; dv_g(x) - \int _{M}
|u^h(0,x)-u(0,x)| \; dv_g(x),
$$
and, if $\tilde{u}$ is replaced by $u$ and the exact differentials
in time-space are used, then the term
$S_{\delta,\epsilon}^h(u,u^h)$ is expected to converge to zero.

Finally, we obtain the basic approximation inequality which is derived as  
 a lower bound for the term $\Ruuh$.

\begin{proposition}[Basic approximation inequality]  
The $L^1$ distance between the approximate and the exact solution satisfies 
\label{PBAI}
$$
\begin{aligned}
& \int_M  |u^h(T,x)-u(T,x)| \; dv_g(x)
\\
& \leq C \; \int_{M} |u^h(0,x)-  u(0,x)|dv_g(x)
\\
& \quad + C \big(1 + \TV_g(u_0)\big) \: (\epsilon + \delta) 
\\
& \quad + C \; \sup_{0 \leq t \leq T} \Big( S_{\delta, \epsilon}^h(u,u^h)
+ E_{\delta, \epsilon}^h(u,u^h)  + E_{\delta,
\epsilon}^h(u^h,u)\Big),
\end{aligned}
$$
where the constant $C>0$ may 
depend on $T$ and also on the metric $g$, but do not depend on
$h$, $\epsilon$, $\tau$, and $\delta$.
\end{proposition}

\begin{proof}
First, we write \eqref{DECOMP2} as
$$
  \Ruuh= R_1 + R_2 + R_3 + R_4,
$$
with obvious notation. For $R_2$, we simply observe that $R_2 \geq 0$.
To estimate $R_1$,  we consider the following
decomposition
$$
\begin{aligned}
& |u^h(T,x')-u(t,x)| 
\\
= & \: |u^h(T,x')-u(T,x')|
\\
& - \Big( |u^h(T,x')-u(T,x')| - |u^h(T,x')-u(T,x)| \Big)
\\
& - \Big( |u^h(T,x')-u(T,x)|-|u^h(T,x')-u(t,x)| \Big)
\\
\geq & \: |u^h(T,x')-u(T,x')| - |u(T,x)- u(T,x')| - |u(t,x)-
u(T,x)|.
\end{aligned}.
$$
Then, using this decomposition in the expression $R_1$, we have
$$
\begin{aligned}
R_1 
=& \int_{Q}\int_{M} |u^h(T,x')-u(t,x)| \: \phi(t,x;T,x')
dv_g(x') dv_g(x) dt
\\
\geq &\frac{1}{2} \int_{M} |u^h(T,x')-u(T,x') | dv_g(x')
\\
&- \frac{1}{2}\Big( \big(C + C \, \TV_g(u_0)\big)\: \epsilon +
 \big( C + C \, \TV_g(u_0)\big)\: \delta \Big).
\end{aligned}
$$

3. Analogously to item 2, we obtain
$$
\begin{aligned}
R_3  
\geq & -\frac{1}{2}\int_{M} |u^h(0,x')-u(0,x')| \; dv_g(x')
\\
&- \frac{1}{2}\Big( \big(C+C \, \TV_g(u_0)\big)\: \epsilon +
 \big(C+C \, \TV_g(u_0)\big)\: \delta \Big),
\end{aligned}
$$
and
$$
\begin{aligned}
 R_4 
 \geq & - \iint_{Q} \rho_{\delta}(t') |u^h(t',x')-
u(t',x')| \; dv_g(x')dt'
\\
&- \frac{1}{2}\Big( \big(C+C \, \TV_g(u_0)\big)\: \epsilon +
 \big(C + C \, \TV_g(u_0)\big)\: \delta \Big).
\end{aligned}
$$

4. Hence adding all these inequalities, we get
$$
\begin{aligned}
& 2 \Ruuh 
\\
&  \geq \int_{M} |u^h(T,x')-u(T,x')| \; dv_g(x')- \int_{M}
|u^h(0,x')-u(0,x')| \; dv_g(x')
\\
& \quad - 2 \iint_{Q} \rho_{\delta}(t')\; |u^h(t',x')- u(t',x')| \;
dv_g(x')dt'
\\
& \quad - 3 \Big( \big(C +C \, \TV_g(u_0)\big)\: \epsilon +
 \big(C + C \, \TV_g(u_0)\big)\: \delta \Big).
\end{aligned}
$$

5. Finally, by a simple algebraic manipulation we deduce from the above inequality that 
$$
\begin{aligned}
\int_{M}& |u^h(T,x')-u(T,x')| \: dv_g(x') &
\\
& \leq A + 2\iint_{Q} \rho_{\delta}(t')
|u^h(t',x')-u(t',x')| \; dv_g(x')dt',
\end{aligned}
$$
where
$$
\begin{aligned}
A=  &\int_{M} |u^h(0,x')-u(0,x')| \; dv_g(x') \\
 &+ 3 \: \Big( \big(C_1(T)+C_1'(T) \, \TV_g(u_0)\big)\: \epsilon +
 \big(\tilde{C_1}(T)+\tilde{C_1'}(T) \, \TV_g(u_0)\big)\: \delta \Big)
\\
& + 2 \, \big( S_{\delta, \epsilon}^h(u,u^h) +
E_{\delta, \epsilon}^h(u,u^h)  + E_{\delta,
\epsilon}^h(u,u^h)\big).
\end{aligned}
$$
Then, applying the Gronwall's inequality, we get
$$
 \int_{M} |u^h(T,x')-u(T,x')| \; dv_g(x') \leq 3 \;  A.
$$
\end{proof}

\subsection{Dealing with the lack of symmetry}

In this subsection we estimate the lack of symmetry in the term
$\Suuh$.

\begin{proposition}[Estimate of $\Suuh$] The following inequality holds
\label{PSUUH}
$$\label{S}
  \begin{aligned}
  \Suuh \leq & C \, \big(1 + \|u_0\|_{L^\infty} \big)\:
  \frac{h}{\epsilon} \; |M| \\
  &+ C \, \big(1 + TV_g(u_0)\big)\:
  \big(\frac{\tau}{\delta}+ \frac{h}{\epsilon} + \epsilon\big),
\end{aligned}
$$
where the constant $C>0$ may depend on $T$ and the metric $g$, but do not
depend on $h$, $\epsilon$, $\tau$, and $\delta$.
\end{proposition}

\begin{proof}  {\bf Step 1.} 
From \eqref{DECOMP3} we write
$$
  \Suuh= S_1 + S_2,
$$
with obvious notation. Consider the decomposition $S_1= S'_1 + S''_1$:
$$
\begin{aligned}
S'_1=  \iint_{Q_T}  \iint_{Q_T}&\Big( |\tilde{u}(t,x)-u^h(t',x')|-
 |u(t,x)-u^h(t',x')| \Big) \\
&  \tilde{\psi}_{\epsilon}(x;x') \: \del_t
\rho_{\delta}(t;t') \,
   dv_g(x')dt' \; dv_g(x) dt,
\end{aligned}
$$
$$
\begin{aligned}
S''_1= \iint_{Q_T}  \iint_{Q_T}  & |u(t,x)-u^h(t',x')| \:
\del_t\rho_{\delta}(t;t') \\
& \Big( \tilde{\psi}_{\epsilon}(x;x')-
\tilde{\psi}'_{\epsilon}(x;x') \Big) \, dv_g(x') dt' \; dv_g(x)
dt.
\end{aligned}
$$
Using \eqref{PAF1},\eqref{PAF2} and the definition of $\tilde{u}$,
it follows that
$$
\begin{aligned}
|S'_1|&\leq \iint_{Q_T} |\tilde{u}(t,x)- u(t,x)| \int_M
\tilde{\psi}_{\epsilon}(x;x') dv_g(x') \int_0^T |
\del_t\rho_{\epsilon}(t;t')| dt' \: dv_g(x) dt
\\
&\leq T \: {C  \over \delta} \sup _{t\in [0,T]} \| \tilde{u} (t) -
u(t)\|_{L^1(M; g)}
 \sup_{K\in \TT^h} \dashint_{\dK} \int_M \psi_{\epsilon} (x;x') dv_g(x') d\Gamma(x)
\\
&\leq C \, T \, {\tau  \over \delta} \, \big(C + C\, \TV_g(u_0) \big).
\end{aligned}
$$
On the other hand, to estimate $S''_1$ we  integrate  by parts
$$
\begin{aligned}
S''_1=& - \iint_{Q_T}  \iint_{Q_T} \sgn(u(t,x)-u^h(t',x')) \:
u_t(t,x) \: \rho_{\delta}(t;t') 
\\
& \hspace{150pt} \Big( \tilde{\psi}_{\epsilon}(x;x')-
\tilde{\psi}'_{\epsilon}(x;x') \Big) \, dv_g(x') dt' \; dv_g(x) dt
\\
&+
 \int_M  \iint_{Q_T}  |u(T,x)-u^h(t',x')| \;  \rho_{\delta}(T;t')
\\
& \hspace{150pt}  \Big(\tilde{\psi}_{\epsilon}(x;x')-
\tilde{\psi}'_{\epsilon}(x;x') \Big) \, dv_g(x') dt' \: dv_g(x)
\\
&- \int_M  \iint_{Q_T}  |u(0,x)-u^h(t',x')| \; \rho_{\delta}(t')
\\
& \hspace{150pt} \Big( \tilde{\psi}_{\epsilon}(x;x')-
\tilde{\psi}'_{\epsilon}(x;x') \Big) \, dv_g(x') dt' \: dv_g(x)
\\
\leq & \iint_{Q_T}  \iint_{Q_T} |u_t(t,x)| \: \rho_{\delta}(t;t')
\: \big| \tilde{\psi}_{\epsilon}(x;x')-
\tilde{\psi}'_{\epsilon}(x;x') \big| \, dv_g(x') dt' \; dv_g(x) dt
\\
&+
 \int_M  \iint_{Q_T}  \Big(|u(T,x)-u^h(t',x')| \rho_{\delta}(T;t') + |u(0,x)-u^h(t',x')| \;
 \rho_{\delta}(t')\Big)
\\
& \hspace{150pt} \big|\tilde{\psi}_{\epsilon}(x;x')-
\tilde{\psi}'_{\epsilon}(x;x') \big| \, dv_g(x') dt' \: dv_g(x).
\end{aligned}
$$
Hence, we conclude that 
$$
|S''_1| \leq  C \:{h\over \epsilon} \Big( T \; \big(C
+ C \, \TV_g(u_0) \big) + |M| \: \big(C + C
\: \|u_0\|_{L^\infty} \big)\Big).
$$
Here, with some abuse of notation we have written $u_t\, dv_g(x')$ to denote the integration with respect 
to the measure $u_t$.  

\

\noindent{\bf Step 2.} Finally, in order to estimate $S_2$ we observe that  
$$
  \begin{aligned}
  S_2 = & \iint _{Q_T}\iint _{Q_T} \sgn(u(t,x)- u^h(t',x'))\\
&   \Big( \big(\big\langle \sigma_{x'}(f_x(u(t,x))-f_{x'}(u(t,x)),\nabla_g
\psi_{\epsilon} \big\rangle_g + (\dive' f)_{x'}(u(t,x))
\psi_{\epsilon} \big)
\tilde{\rho}_{\delta} \\
& - \big(\big\langle \sigma_{x'} ( f_x(u^h(t',x')) - f_{x'}(u^h(t',x'),),\nabla_g
\psi_{\epsilon} \big\rangle_g + (\dive f)_x(u^h(t',x'))
\psi_{\epsilon} \big)
\tilde{\rho}_{\delta} \\
& + \langle f_{x'}(u(t,x)) - f_{x'}(u^h(t',x')),\nabla_g
\psi_{\epsilon} \rangle_g  \big(\tilde{\rho}_{\delta}-\tilde{\rho}_{\delta}' \big)   \\
& - \big(\dive' f\big)_{x'}(u(t,x)) \big(\phi^h - \phi^{h'}\big)
\Big) \, dv_g(x') dt' dv_g(x) dt,
\end{aligned}
$$
where $\sigma_{x'}(f_x(u(t,x))$ is the parallel transport of the vector $f_x(u(t,x))$
from the point $x$ to $x'$. 
We use that $f$ is a smooth vector field on $M$. For $x \in M$ fixed, $u \in \RR$ fixed and $\epsilon>0$
sufficiently small, using the Landau notation $O(\cdot)$ we write
$$
\begin{aligned}
&\langle  \sigma_{x'}(f_x(u)) - f_{x'}(u),\nabla_g \psi_\epsilon(x;x') \rangle_g +
\big(\dive' f\big)(u,x') \: \psi_\epsilon(x;x')
\\
&= \langle \big(\nabla_g' f\big)_{x'}(u)\: k(x;x') ,\nabla_g
\psi_\epsilon(x,x') \rangle_g + \big(\tr \nabla_g' f \big)(u,x')
\: \psi_\epsilon(x,x') + O(d_g^2(x,x'))
\\
&= \langle \big(\nabla_g' f\big)_{x'}(u), \nabla_g \big( k(x;x') \:
\psi_\epsilon(x;x')\big) \rangle_g + O(\epsilon^2),
\end{aligned}
$$
where $k(x;x')$ is the tangent vector at $x$ of the minimizing
geodesic from $x$ to $x'$. Analogously, we have
$$
\begin{aligned}
&\langle  \sigma_{x'}(f_x(u^h)) - f(u^h,x'),\nabla_g \psi_\epsilon(x;x')
\rangle_g + \big(\dive f\big)(u^h,x) \: \psi_\epsilon(x;x')
\\
&= \langle \big(\nabla_g' f\big)(u^h,x'), \nabla_g \big( k(x;x')
\: \psi_\epsilon(x;x')\big) \rangle_g + O(\epsilon^2).
\end{aligned}
$$
Now, denoting the Lipschitz $(1,1)-$tensor field $I\!\!I$ as
$$
I\!\!I(t,x;t',x'):= \big(\nabla_g' F_{x'} \big)(u(t,x),u^h(t',x')),
$$
we have $|S_2| \leq |S'_2| + |S''_2| + O(\epsilon)$, where
$$
  \begin{aligned}
  S'_2 &= \iint _{Q_T}\iint _{Q_T} \big\langle I\!\!I(t,x;t',x'),  \nabla_g \big( k(x;x') \:
\psi_\epsilon(x;x')\big) \big\rangle_g \:  \tilde{\rho}_{\delta}
\,
  dv_g(x') dt'dv_g(x) dt,
  \end{aligned}
$$
$$
  \begin{aligned}
 S''_2 = \iint_{Q_T}& \iint_{Q_T} \sgn(u-u^h)
  \Big( \big\langle f_{x'}(u(t,x)) - f_{x'}(u^h(t',x')),\nabla_g
\psi_{\epsilon} \big\rangle_g
\big(\tilde{\rho}_{\delta}-\tilde{\rho}_{\delta}' \big)
\\
& \hspace{30pt}- \big(\dive' f\big)_{x'}(u(t,x)) \big(\phi^h -
\phi^{h'} \big) \Big) \, dv_g(x') dt' dv_g(x) dt.
  \end{aligned}
$$
The term $S_2''$ is estimated as the final step of item $2$, we
focus on $S_2'$ term. Applying the Gauss-Green formula with
respect to the $x$ variable, it follows that
$$
  \iint _{Q_T} \iint _{Q_T}  \big\langle I\!\!I(t,x';t',x'),
  \nabla_g \big( k(x;x') \:
\psi_\epsilon(x;x')\big) \big\rangle_g \tilde{\rho}_{\delta} \,
  dv_g(x') dt' dv_g(x) dt  =0.
$$
Therefore, subtracting the above expressions from $S_2'$, we
obtain
$$
  |S_2'| \leq C \; T \; \big( C +C \, \TV_g(u_0)\big) \:
  \epsilon,
$$
where, we have used the fact that, due to the compactness of $M$ and 
the regularity of the flux function, 
the function $\nabla_g f(u,x)$ is uniformly Lipschitz continuous for $u$ in a compact set. 
Combining this result with the estimation of $S''_2$, that is
$$
  |S_2''| \leq C \; T \; \Big( \frac{\tau}{\delta} \big(C + C \, \TV_g(u_0)\big)
  + \frac{h}{\epsilon} \big(C + C \|u_0\|_{L^\infty} \big) |M|_g
  \Big),
$$
we complete the proof of the proposition.
\end{proof}
%

\subsection{Entropy production for the exact solution}

We now  consider the approximate entropy dissipation associated with the exact solution. 

\begin{proposition}[Estimate of the quantity $\Euuh$]
 The following inequality holds 
\label{PEUUH}
$$
   \Euuh \leq C \, |M|_g \, \big(1+ \|u_0\|_{L^\infty}\big) +
    C \, \big( 1 + \TV_g(u_0)\big)\Big) \: \frac{h}{\epsilon},
$$
where $C>0$ may depend on $T$ and the metric $g$, but do not
depend on $h$, $\epsilon$, $\tau$, and $\delta$.
\end{proposition}

\begin{proof} 1. For each $c \in \RR$ and in the sense of distributions we have
$$
  U_t(u(t,x),c) + \dive F_x(u(t,x),c) + \sgn(u(t,x)-c) (\dive f)(c,x) \leq
  0.
$$
Now, we set $c=u^h(t',x')$ for each $(t',x') \in [0,T] \times M$
fixed. Since $u$ is an entropy solution to \eqref{CON1}, for
$n=0,1,\ldots$, we have
$$
\begin{aligned}
& \int_M  \Big(U\big(u(t_{n+1},x),u^h(t',x')\big) - U\big(u(t_n, x),
u^h(t',x')\big)\Big) \psi_{\epsilon}(x;x')dv_g(x)
\\
&- \int_{t_n}^{t_{n+1}}\int_M \langle F_x(u(t,x),u^h(t',x')),
\nabla_g \psi_{\epsilon}(x;x') \rangle_g \, dv_g(x) dt
\\
&+ \int_{t_n}^{t_{n+1}}\int_M \sgn(u(t,x)-u^h(t',x')) (\dive
f)(u^h(t',x'),x) \: \psi_{\epsilon}(x;x') \, dv_g(x) dt \leq 0.
\end{aligned}
$$

2. Next, we multiply this inequality by
$\tilde{\rho}_\delta(t;t')=\rho_{\delta}(t_{n+1};t')$ and summing
the first term in time, it follows that
$$
\begin{aligned}
 -\sum _{n=0}^{n_T-1} & \int_M U(u(t_n,x),u^h)\Big(\rho_{\delta}(t_{n+1};t')
 - \rho_{\delta}(t_n;t')\Big)\psi_{\epsilon}(x;x')dv_g(x)
\\
&-\iint_{Q_T} \langle F_x(u(t,x),u^h),\nabla_g
\psi_{\epsilon}(x;x')\rangle \tilde{\rho}_{\delta}(t;t') \;
dv_g(x) dt
\\
&+\iint_{Q_T} \sgn(u(t,x)-u^h) (\dive f)_x(u^h) \:
\psi_{\epsilon}(x;x') \tilde{\rho}_{\delta}(t;t') \; dv_g(x) dt
\\
&+  \int_M U(u(T,x),u^h)\rho_{\delta}(T;t')
\psi_{\epsilon}(x;x')dv_g(x)
\\
& -\int_M U((u(0,x),u^h)\rho_{\delta}(t')
\psi_{\epsilon}(x;x')dv_g(x)\leq 0.
\end{aligned}
$$
By the definition \eqref{A.E.S.} of $\tilde{u}$ we have the identity
$$
\begin{aligned}
\int_M & U(u(t_n,x),u^h)\Big(\rho_{\delta}(t_{n+1};t') -
\rho_{\delta}(t_n;t')\Big)\psi_{\epsilon}(x;x')dv_g(x)
\\ &= \int_{t_n}^{t_{n+1}}\int_M U(\tilde{u}(t,x),u^h) \del_t \rho_{\delta}(t;t') \psi_{\epsilon} (x;x') \,  dv_g(x)
dt.
\end{aligned}
$$
Therefore, $\Euuh$ is bounded above by
$$
\begin{aligned}
   \iint_{Q_T} \iint_{Q_T}& U(\tilde{u}(t,x),u^h) \del_t
\rho_{\delta}(t;t')
\\
&  \Big(\psi_{\epsilon} (x;x') - \tilde{\psi}_{\epsilon}
(x;x') \Big)\, dv_g(x) dt \; dv_g(x') dt'.
\end{aligned}
$$

3. By integrating by parts the above equation
with respect to $t$, it follows that
$$
\begin{aligned}
I := &\iint_{Q_T}\iint_{Q_T}U(\tilde{u}(t,x),u^h) \del_t
\rho_{\delta}(t;t') \big(\psi_{\epsilon}- \tilde{\psi}_{\epsilon}
\big)\, dv_g(x) dt \; dv_g(x') dt'
\\
=& -\sum _{n=0}^{n_T-1} \int_M \iint_{Q_T}
\Big(U(u(t_{n+1},x),u^h)- U(u(t_n,x),u^h)\Big) \; (\psi_{\epsilon}
- \tilde{\psi}_{\epsilon})
 \\   &   \hspace{140pt} 
\rho_{\delta}(t_{n+1};t') dv_g(x')dt' \; dv_g(x)
\\
&+ \int_M\iint_{Q_T} U(u(T,x),u^h)\big(\psi_{\epsilon}-
\tilde{\psi}_{\epsilon} \big) \rho_{\delta} (T;t')\, dv_g(x') dt'
\; dv_g(x)
\\
& -\int_M\iint_{Q_T} U(u(0,x),u^h)\big(\psi_{\epsilon}-
\tilde{\psi}_{\epsilon} \big) \rho_{\delta} (t')\, dv_g(x') dt'\;
dv_g(x), 
\end{aligned}
$$
and thus 
$$
\begin{aligned}
I \leq 
& \sum _{n=0}^{n_T-1} \tau \int_M \iint_{Q_T} |u_t| \:
\rho_{\delta}(t_{n+1}; t') \: |\psi_{\epsilon}(x;x')-
\tilde{\psi}_{\epsilon}(x;x')| dv_g(x')dt' \; dv_g(x)
\\
& +\int_{M} \iint_{Q_T} \Big( |u(T,x) - u^h(t',x')|
\rho_{\delta}(T;t') + |u(0,x) - u^h(t',x')| \rho_{\delta}(t')
\Big)
\\
& \hspace{50pt}  |\psi_{\epsilon}(x;x')-
\tilde{\psi}_{\epsilon}(x;x')| \: dv_g(x')dt' \; dv_g(x)
\\
\leq 
& C \; \frac{h}{\epsilon} \Big( T \; \big(C +C \, \TV_g(u_0)\big) + |M|_g \:
\big(C + C \, \|u_0\|_{L^\infty}\big) \Big).
\end{aligned}.
$$
\end{proof}

\section{Entropy production for the approximate solutions}

It remains to control the approximate entropy dissipation
form for the approximate solution.

\begin{proposition}[Estimate of the quantity $\Euhu$]
The following inequality holds 
\label{PEUHU}
$$
\begin{aligned}
  \Euhu \leq & C \, \big(1+ \|u_0\|_{L^\infty} \big)\:
  \big(\frac{h}{\epsilon} + \tau + \epsilon \big) |M|_g
\\
  &+ C \, \big(1 + \|u_0\|_{L^2(M;g)} \big) \:
\frac{h^{1/2}}{\epsilon} \; |M|_g^{1/2},
\end{aligned}
$$
where $C>0$ may depend on $T$ and the metric $g$, but do not
depend on $h$, $\epsilon$, $\tau$, and $\delta$.
\end{proposition}

Once this estimate is established we can complete the proof of the main theorem, as follows.

\begin{proof}[Proof of Theorem~\ref{MT}]
1. First, by \eqref{IN1} there exists $\gamma_1>0$, such that
$\tau \leq \gamma_1 \; h$. Moreover, without loss of generality we
can take $\delta = \epsilon$. Therefore, combining 
Propositions \ref{PBAI}-\ref{PEUHU} together and denoting 
$$
  \begin{aligned}
  &A_0(T):= C \, \big(1 + \|u_0\|_{L^\infty}\big), 
\qquad 
A_1(T):= C \, \big(1+\TV_g(u_0)\big)\\
  &A_2(T):= C \, \big(1 + \|u_0\|_{L^2(M;g)} \big),
  \end{aligned}
$$
we obtain
$$
  \|u^h(T)-u(T)\|_{L^1(M; g)} \leq \frac{C}{2} \: \Big(I\!\!L_{-1} \: \epsilon^{-1}
  +2 \; I\!\!L_{0}+ I\!\!L_{1} \: \epsilon \Big),
$$
where
$$
\begin{aligned}
&I\!\!L_{-1}:= A_0(T) \:|M|_g \: h
 + A_1(T) \: h + A_2(T) \; |M|_g^{1/2} \: h^{1/2},
\\
&I\!\!L_{0}:= A_0(T) \:|M|_g \: h + \|u^h(0)-u(0)\|_{L^1(M;g)},
\\
& I\!\!L_{1}:= A_0(T) \; |M|_g + A_1(T). 
\end{aligned}
$$
Then, minimizing with respect to
$\epsilon$, we obtain
$$
  \|u^h(T)-u(T)\|_{L^1(M;g)} \leq C \: \Big( \sqrt{ I\!\!L_{-1} \: I\!\!L_{1}}
  + I\!\!L_{0} \Big).
$$

2. Next, proceeding as in the proof of Lemma \ref{KL} and by
 \eqref{PCF}, that is, $u^h(0,x) = u^0_K$, we have
$$
\aligned 
  \|u^0_K-u_0\|_{L^1(M;g)} 
  & \leq
  \int_M \dashint_{K}|u_0(z)-u_0(x)| dv_g(z) \; dv_g(x)
  \\
  & \leq \|\nabla_g u_0\|_{L^1(M;g)} \: h = \TV_g(u_0) \: h.
\endaligned
$$
Consequently, it follows that
$$
  \begin{aligned}
  & \|u^h(T)-u(T)\|_{L^1(M;g)} 
  \\
  &\leq \Big( A_0(T) |M|_g + A_1(T)
  \Big) \: C \;h
  \\
  & \quad + \Big(A_0(T) |M|_g^{1/2} + \big(A_0(T) \, A_1(T)\big)^{1/2}
  \Big) \: C \; |M|_g^{1/2} \; h^{1/2}
  \\
  & \quad + \Big( \big(A_0(T) \, A_2(T) \big)^{1/2} |M|_g^{1/2} + \big(A_1(T) \, A_2(T)\big)^{1/2}
  \Big) \: C \; |M|_g^{1/4} \; h^{1/4},
  \end{aligned}
$$
which completes the proof of Theorem~\ref{MT}.
\end{proof}

\begin{proof}[Proof of Proposition~\ref{PEUHU}] 1. Fix $K \in \TT^h$ and $e \in \partial K$.
For $(t,x) \in [0,T] \times M$ fixed, we set $c =
u(t,x)$ and $u^h(t',x') = u^n_K$, for $(t',x') \in [t'_n,t'_{n+1})
\times M$, $(n= 0,1,\ldots)$ and we define 
$$
  \tilde{\psi}'_{\epsilon,e}(x;x'):= \dashint_e
  \psi_\epsilon(x;y') \; d\Gamma_g(y').
$$
Therefore, by Definition \ref{PAF2} and analogously to Lemma
\ref{KL1}, it follows that
$$
\aligned
& \tilde{\psi}'_{\epsilon}(x;x') 
   = \sum_{e \in \dK}
  \frac{|e|}{p_K} \: \tilde{\psi}'_{\epsilon,e}(x;x'),
\\ 
& \int_M |\tilde{\psi}'_{\epsilon}(x;x') -
  \tilde{\psi}'_{\epsilon,e}(x;x')|\; dv_g(x)  \leq C \; \frac{h}{\epsilon},
  \endaligned 
$$
where the positive constant $C$ does not depend on $h,\epsilon>0$.
Now, we write the local entropy inequality \eqref{FVM.7} for $K$
and, since it is also valid for $K_e$, we obtain respectively
$$
\begin{aligned}
\frac{|K|}{p_{K}} \Big(U(u_{K,e}^{n+1},c)& - U(u_K^n,c) \Big)
\\
&+ \tau \Big( F_{e,K}(u_K^n,u_{K,e}^n,c) - F_{e,K}(u_K^n,
u_K^n,c)\Big) \leq \frac{|K|} {p_{K}} \dnnkee,
\end{aligned}
$$
$$
\begin{aligned}
\frac{|K_e|}{p_{K_e}} \Big(U(u_{K_e,e}^{n+1},c)& - U(u_{K_e}^n,c)
\Big)\\
& + \tau \Big( F_{e,K_e}(u_{K_e}^n,u_{K}^n,c) -
F_{e,K_e}(u_{K_e}^n, u_{K_e}^n,c)\Big) \leq \frac{|K_e|} {p_{K_e}}
D_{K_e,e}^{n+1}.
\end{aligned}
$$
We sum the two above inequalities and from \eqref{FVM.4} and
\eqref{FVM.5}, we obtain
\be
\begin{aligned} \label{1000}
\frac{|K|}{p_{K}}& \Big(U(u_{K,e}^{n+1},c) - U(u_K^n,c) \Big) +
\frac{|K_e|}{p_{K_e}} \Big(U(u_{K_e,e}^{n+1},c) - U(u_{K_e}^n,c)
\Big)\\
&+ \tau \: \dashint_{e} \langle
F_{y'}(u_{K_e}^n,c)-F_{y'}(u_{K}^n,c), \nek(y')\rangle_g \,
    d\Gamma_g(y')
\\
&\leq \frac{|K|} {p_{K}} \dnnkee + \frac{|K_e|}
{p_{K_e}}D_{K_e,e}^{n+1}.
\end{aligned}
\ee

2. We multiply inequality \eqref{1000} by $|e| \;
\tilde{\psi}'_{\epsilon,e}$ and sum over all $e \in \partial K$
and $K \in \TT^h$:
\be
\begin{aligned} \label{2000}
&\sum_{\substack{e\in \dK \\ K \in \TT^h}} \frac{|K|}{p_{K}}
U(u_{K,e}^{n+1},c) \; |e| \; \tilde{\psi}'_{\epsilon,e} - \sum_{K
\in \TT^h} |K| \; U(u_K^n,c) \; \tilde{\psi}'_{\epsilon}
\\
& - \tau \sum_{K \in \TT^h} \int_{\partial K} \langle
F_{y'}(u_{K}^n,c), \nek(y')\rangle_g \,
\tilde{\psi}'_{\epsilon,e}(x,x') \;
    d\Gamma_g(y')
\\
&\leq \sum_{\substack{e\in \dK \\ K \in \TT^h}} \frac{|K|} {p_{K}}
\: |e| \; \tilde{\psi}'_{\epsilon,e} \; \dnnkee,
\end{aligned}
\ee
where we have used 
$$
  \sum_{\substack{e\in \dK \\ K \in \TT^h}} \frac{|K_e|}{p_{K_e}}
  U(u_{K_e,e}^{n+1},c)\; \tilde{\psi}'_{\epsilon,e} = \sum_{\substack{e\in \dK \\ K \in \TT^h}} \frac{|K|}{p_{K}}
  U(u_{K,e}^{n+1},c) \; \tilde{\psi}'_{\epsilon,e},
$$
$$
  \sum_{\substack{e\in \dK \\ K \in \TT^h}} \frac{|K_e|}{p_{K_e}}
  U(u_{K_e,e}^{n},c)\; \tilde{\psi}'_{\epsilon,e} = \sum_{\substack{e\in \dK \\ K \in \TT^h}} \frac{|K|}{p_{K}}
  U(u_{K,e}^{n},c) \; \tilde{\psi}'_{\epsilon,e},
$$
$$
\begin{aligned}
  \sum_{\substack{e\in \dK \\ K \in \TT^h}} \dashint_{e}
\langle F_{y'}(u_{K_e}^n,c), &\nek(y')\rangle_g \;
\tilde{\psi}'_{\epsilon,e} \,
    d\Gamma_g(y')
        \\
     &= - \sum_{\substack{e\in \dK \\ K \in \TT^h}} \dashint_{e}
\langle F_{y'}(u_{K}^n,c), \nek(y')\rangle_g \;
\tilde{\psi}'_{\epsilon,e} \,
    d\Gamma_g(y').
\end{aligned}
$$
Since $u_{K}^{n+1}$ is a convex combination of $u_{K,e}^{n+1}$ and
the Kruzkov's entropy $U$ is convex, we have by Jensen's
inequality
$$
  \sum_{K \in \TT^h} |K| \; U(u_{K}^{n+1})\; \tilde{\psi}'_{\epsilon} \leq
  \sum_{\substack{e\in \dK \\ K \in \TT^h}} \; |K| \; \frac{|e|}{p_{K}}
  \; U(u_{K,e}^{n+1}) \; \tilde{\psi}'_{\epsilon}.
$$
Therefore, from \eqref{2000} we obtain 
\be
\begin{aligned} \label{3000}
&\sum_{K \in \TT^h} |K| \Big( U(u_{K}^{n+1},c) - U(u_K^n,c) \Big)
\; \tilde{\psi}'_{\epsilon}
\\
& \quad - \tau \sum_{K \in \TT^h} \int_{\partial K} \langle
F_{y'}(u_{K}^n,c), \nek(y')\rangle_g \, \psi_{\epsilon}(x;y') \;
d\Gamma_g(y')
\\
&\leq \sum_{\substack{e\in \dK \\ K \in \TT^h}} \frac{|K|} {p_{K}}
\: |e| \; \tilde{\psi}'_{\epsilon,e} \; \dnnkee
+ \sum_{\substack{e\in \dK \\ K \in \TT^h}} \frac{|K|} {p_{K}} \:
|e| \; U(u_{K,e}^{n+1},c) \:
\big(\tilde{\psi}'_{\epsilon}-\tilde{\psi}'_{\epsilon,e}\big)
\\
& \quad + \tau \sum_{\substack{e\in \dK \\ K \in \TT^h}} \int_{e}
\langle F_{y'}(u_{K}^n,c), \nek(y')\rangle_g \, \Big(
\tilde{\psi}'_{\epsilon,e}(x;x')-\psi_{\epsilon}(x;y')\Big) \;
d\Gamma_g(y').
\end{aligned}
\ee

3. Applying Gauss-Green's formula it follows that
$$
\begin{aligned}
&\int_{\partial K} \langle F_{y'}(u_{K}^n,c),
  \nek(y')\rangle_g \, \psi_{\epsilon}(x;y') \; d\Gamma_g(y')
\\
 & = \int_K \big(\dive' F_{x'}\big)(u_{K}^n,c)
 \psi_{\epsilon}(x;x') \; dv_g(x')
  + \int_K \langle F_{x'}(u_{K}^n,c), \nabla_g'
\psi_{\epsilon}(x;x')\rangle_g \; dv_g(x').
\end{aligned}
$$
Then, from \eqref{3000} we deduce that 
$$
\begin{aligned} 
&\sum_{K \in \TT^h} |K| \Big( U(u_{K}^{n+1},c) - U(u_K^n,c) \Big) \; \tilde{\psi}'_{\epsilon}
\\
& \quad - \int_{t'_n}^{t'_{n+1}} \int_{M} \langle
       F_{x'}(u^h(t',x'),c), \nabla_g' \psi_{\epsilon}(x;x')\rangle_g \; dv_g(x')
\\
 & \quad - \int_{t'_n}^{t'_{n+1}} \int_{M} \big(\dive' F_{x'}\big)(u^h(t',x'),c)
 \psi_{\epsilon}(x;x') \; dv_g(x')
\\
 & \leq \sum_{\substack{e\in \dK \\ K \in \TT^h}} \frac{|K|} {p_{K}}
\: |e| \; \tilde{\psi}'_{\epsilon,e} \; \dnnkee
+ \sum_{\substack{e\in \dK \\ K \in \TT^h}} \frac{|K|} {p_{K}} \:
|e| \; U(u_{K,e}^{n+1},c) \:
\big(\tilde{\psi}'_{\epsilon}-\tilde{\psi}'_{\epsilon,e}\big)
\\
& \quad + \int_{t'_n}^{t'_{n+1}} \sum_{\substack{e\in \dK \\ K \in
\TT^h}} \int_{e} \langle F_{y'}(u_{K}^n,c), \nek(y')\rangle_g \,
\Big(\tilde{\psi}'_{\epsilon,e}(x;x')-\psi_{\epsilon}(x;y')\Big)
\; d\Gamma_g(y').
\end{aligned}
$$
By algebraic manipulation, we see that the expression
$$
\begin{aligned} 
&\sum_{K \in \TT^h} |K| \Big( U(u_{K}^{n+1},c) - U(u_K^n,c) \Big)
\; \tilde{\psi}'_{\epsilon}
\\
& \quad - \int_{t'_n}^{t'_{n+1}} \int_{M} \langle
F_{x'}(u^h(t',x'),c), \nabla_g' \psi_{\epsilon}(x;x')\rangle_g \;
dv_g(x')dt'
\\
 & \quad + \int_{t'_n}^{t'_{n+1}} \int_{M} \sgn(u^h(t',x')-c) \; \big(\dive' f \big)(c,x')
 \psi_{\epsilon}(x;x') \; dv_g(x')dt'
\end{aligned}
$$
is bounded above by 
$$
\begin{aligned}
 & \leq \int_{t'_n}^{t'_{n+1}} \int_{M} \sgn(u^h(t',x')-c) \; \big(\dive' f \big)(u^h(t',x'),x')
 \psi_{\epsilon}(x;x') \; dv_g(x')dt'
\\
  & \quad + \sum_{\substack{e\in \dK \\ K \in \TT^h}} \frac{|K|} {p_{K}}
\: |e| \; \tilde{\psi}'_{\epsilon,e} \; \dnnkee
+ \sum_{\substack{e\in \dK \\ K \in \TT^h}} \frac{|K|} {p_{K}} \:
|e| \; U(u_{K,e}^{n+1},c) \:
\big(\tilde{\psi}'_{\epsilon}-\tilde{\psi}'_{\epsilon,e}\big)
\\
& \quad + \int_{t'_n}^{t'_{n+1}} \sum_{\substack{e\in \dK \\ K \in
\TT^h}} \int_{e} \langle F_{y'}(u_{K}^n,c), \nek(y')\rangle_g \,
\Big(\tilde{\psi}'_{\epsilon,e}(x;x')-\psi_{\epsilon}(x;y')\Big)
\; d\Gamma_g(y').
\end{aligned}
$$
Now, we multiply this inequality by
$\tilde{\rho}'_\delta(t;t')=\rho_{\delta}(t;t'_{n+1})$ and summing
with respect to time variable, i.e. $t'$, we obtain that the expression 
$$
\begin{aligned} 
& - \sum _{n=0}^{n_T-1} \int_{M} U(u^n_K,c)
\Big(\rho_{\delta}(t;t'_{n+1}) - \rho_{\delta}(t;t'_{n}) \Big) \;
\tilde{\psi}'_{\epsilon}(x;x') \; dv_g(x')
\\
& - \iint_{Q_T} \langle F_{x'}(u^h(t',x'),c), \nabla_g'
\psi_{\epsilon}(x;x')\rangle_g \; \tilde{\rho}'_\delta(t;t') \;
dv_g(x')dt'
\\
 & + \iint_{Q_T} \sgn(u^h(t',x')-c) \; \big(\dive' f \big)(c,x')
 \psi_{\epsilon}(x;x') \; \tilde{\rho}'_\delta(t;t') \; dv_g(x')dt'
\\
&+  \int_M U(u^h(T,x'),c)\rho_{\delta}(t;T) \;
\tilde{\psi}'_{\epsilon}(x;x') \; dv_g(x')
\\
& -\int_M U(u^h(0,x'),c)\rho_{\delta}(t) \;
\tilde{\psi}'_{\epsilon}(x;x') \; dv_g(x')
\end{aligned}
$$
is bounded above by
$$
\begin{aligned} 
\leq 
& \iint_{Q_T} \sgn\big(u^h(t',x')-c\big) \; \big(\dive' f
\big)(u^h(t',x'),x')
 \psi_{\epsilon}(x;x') \; \tilde{\rho}'_\delta(t;t') \; dv_g(x')dt'
\\
  & + \sum _{n=0}^{n_T-1} \sum_{\substack{e\in \dK \\ K \in \TT^h}} \frac{|K|} {p_{K}}
\: |e| \; \tilde{\psi}'_{\epsilon,e} \; \tilde{\rho}'_\delta(t;t')
\; \dnnkee
\\
&+ \sum _{n=0}^{n_T-1} \sum_{\substack{e\in \dK \\ K \in \TT^h}}
\frac{|K|} {p_{K}} \: |e| \; U(u_{K,e}^{n+1},c) \:
\big(\tilde{\psi}'_{\epsilon}-\tilde{\psi}'_{\epsilon,e}\big) \;
\tilde{\rho}'_\delta(t;t')
\\
& + \int_{0}^{T} \hskip-.3cm  \sum_{\substack{e\in \dK \\ K \in \TT^h}}  \hskip-.15cm p_K
\: \frac{|e|}{p_K} \dashint_{e} \langle F_{y'}(u_{K}^n,c),
\nek(y')\rangle_g
       \Big(\tilde{\psi}'_{\epsilon,e}(x;x')-\psi_{\epsilon}(x;y')\Big)
\; \tilde{\rho}'_\delta(t;t') \; d\Gamma_g(y') dt'.
\end{aligned}
$$

4. Following the lines of proof of Proposition \ref{PEUUH}, we can also derive 
the identity
$$
\begin{aligned}
 & \int_M  U(u^n_K,c)\Big(\rho_{\delta}(t;t'_{n+1}) - \rho_{\delta}(t;t'_{n})\Big)
  \tilde{\psi}'_{\epsilon}(x;x') dv_g(x')
\\
 &= \int_{t'_n}^{t'_{n+1}}\int_M U(u^h(t',x'),c) \;
\del_{t'} \rho_{\delta}(t;t') \; \tilde{\psi}'_{\epsilon}(x;x') \,
dv_g(x') dt'.
\end{aligned}
$$
Then, the term $\Euhu$ is bounded above by
\be
\begin{aligned} \label{7000}
& \iint_{Q_T} \Big( \sum _{n=0}^{n_T-1} \sum_{\substack{e\in \dK \\
K \in \TT^h}} \frac{|K|} {p_{K}} \: |e| \;
 \tilde{\rho}'_\delta \; \tilde{\psi}'_{\epsilon,e} \;
 \dnnkee \Big) dv_g(x)dt
\\
& + \iint_{Q_T} \Big(\sum _{n=0}^{n_T-1} \sum_{\substack{e\in \dK \\
K \in \TT^h}} \frac{|K|} {p_{K}} \: |e| \; U(u_{K,e}^{n+1},c) \:
\big(\tilde{\psi}'_{\epsilon}-\tilde{\psi}'_{\epsilon,e}\big) \;
\tilde{\rho}'_\delta(t;t')\Big) dv_g(x)dt
\\
&+  \iint_{Q_T} \int_M \Big(|u^h(T,x')-u(t,x)| \rho_{\delta}(t;T)
+ |u^h(0,x')-u(t,x)| \; \rho_{\delta}(t)\Big) \;
\\
& \hspace{50pt} 
\big|\psi_{\epsilon}(x;x')-\tilde{\psi}'_{\epsilon}(x;x')\big| \;
dv_g(x') \; dv_g(x)dt
\\
 &+ \iint_{Q_T}\int_{0}^{T} \sum_{K \in \TT^h} \int_{K}
  \big|\big(\dive'f \big)(u_{K}^n,x')
  -\big(\dive'f \big)(u_{K}^n,x)\big|
\\
  & \hspace{90pt}  \tilde{\rho}'_\delta(t;t') \;
\psi_{\epsilon}(x;x') \; dv_g(x')dt' \; dv_g(x)dt
\\
& + \iint_{Q_T}\int_{0}^{T} \sum_{K \in \TT^h} \int_{\dK}
\big|F_{y'}(u_{K}^n,c)- \sigma_{y'}(F_{y}(u_{K}^n,c)) \big|_g 
\\
& \hspace{80pt} 
\big|\tilde{\psi}'_{\epsilon}(x;x')-\psi_{\epsilon}(x;y')\big| \;
\tilde{\rho}'_\delta(t;t') \; d\Gamma_g(y') dt'  dv_g(x)dt. 
\end{aligned}
\ee

5. We write \eqref{7000} as $E_1 + E_2 + E_3 + E_4 + E_5$ with 
obvious notation. In order to estimate $E_1$ we recall that
$$
  \dnnkee = U(\unnkee,c) - U(\tunnkee,c),
$$
and
$$
  \unnkee - \tunnkee = - \tau \: \dashint_{K} \big(\dive'
  f\big)(u^n_K,x') dv_g(x').
$$
Therefore, since $U$ is convex from Aleksandrov's theorem it has
second derivative a.e. (see \cite{EG}), and we can write
$$
  \dnnkee = \tau \: \Big(- \partial_u U(\tunnkee,c) \; \dashint_{K}
  \big(\dive'
  f\big)(u^n_K,x') \: dv_g(x') + O(\tau) \Big).
$$
Then, we have
$$
\begin{aligned}
E_1 
& \leq \iint_{Q_T} \int_0^T \sum_{K \in \TT^h} \int_{K}
\big|\big(\dive' f \big)(u^n_K,x')-\big(\dive' f
\big)(u^n_K,x)\big| \; \tilde{\rho}'_\delta \;
\tilde{\psi}'_{\epsilon}\; dv_g(x')dt'
\\
& \quad + C \, T \; |M|_g \, \big(1+\|u_0\|_{L^\infty}\big)
\:\tau
\\
& \leq C' \: T \: |M|_g \: \big(1+
\|u_0\|_{L^\infty}\big)\: (\tau + \epsilon),
\end{aligned}
$$
where we have used the fact that for every compact $K$
the function $\nabla_g \big(\dive f\big)$
  is uniformly bounded in $K \times M$. Now to
estimate $E_2$, we observe that
$$
  \sum _{n=0}^{n_T-1} \sum_{\substack{e\in \dK \\
K \in \TT^h}} \frac{|K|} {p_{K}} \: |e| \; U(u_{K}^{n+1},c) \:
\big(\tilde{\psi}'_{\epsilon}-\tilde{\psi}'_{\epsilon,e}\big) = 0.
$$
Moreover, from \cite{ABL} we recall the uniform bound
$$
  \sum _{n=0}^{n_T-1} \sum_{\substack{e\in \dK \\
K \in \TT^h}} \frac{|K|} {p_{K}} \: |e| \;
|u_{K,e}^{n+1}-u_{K}^{n+1}|^2 \leq \|u_0\|^2_{L^2(M;g)} + C''
$$
for some the constant $C''>0$.

Hence, we have
$$
  \begin{aligned}
   E_2&= \iint_{Q_T} \Big(\sum _{n=0}^{n_T-1} \sum_{\substack{e\in \dK \\
K \in \TT^h}} \frac{|K|} {p_{K}} \: |e| \;
\big(U(u_{K,e}^{n+1},c)-U(u_{K}^{n+1},c)\big)
\\
&\hspace{50pt} 
\big(\tilde{\psi}'_{\epsilon}-\tilde{\psi}'_{\epsilon,e}\big) \;
\tilde{\rho}'_\delta(t;t')\Big) dv_g(x)dt
\\
& \leq \sum _{n=0}^{n_T-1} \sum_{\substack{e\in \dK \\
K \in \TT^h}} \frac{|K|} {p_{K}} \: |e| \;
\big|u_{K,e}^{n+1}-u_{K}^{n+1}\big| \int_M
\big|\tilde{\psi}'_{\epsilon}-\tilde{\psi}'_{\epsilon,e}\big| \;
dv_g(x)
\\
& \leq C \frac{\gamma_1}{\epsilon} \sum _{n=0}^{n_T-1} \sum_{\substack{e\in \dK \\
K \in \TT^h}} \frac{|K|} {p_{K}} \: |e| \;
\big|u_{K,e}^{n+1}-u_{K}^{n+1}\big| \: \tau,
  \end{aligned}
$$
where we have used \eqref{IN1}. Applying Cauchy-Schwartz's 
inequality, we obtain
$$
  \begin{aligned}
   E_2& \leq C \; \frac{\gamma_1}{\epsilon} \; \big( T \;|M|_g)^{1/2} \Big(\sum _{n=0}^{n_T-1} \sum_{\substack{e\in \dK \\
K \in \TT^h}} \frac{|K|} {p_{K}} \: |e| \;
\big|u_{K,e}^{n+1}-u_{K}^{n+1}\big|^2 \: \tau \Big)^{1/2}
\\
& \leq C \; \sqrt{\gamma_1} \;\big( T \;|M|_g)^{1/2} \;
\frac{h^{1/2}}{\epsilon} \: \Big(\|u_0\|^2_{L^2(M;g)} + C(T)
\Big)^{1/2} \:
\\
& \leq C_2 \, \big(1 + \; \|u_0\|_{L^2(M;g)} \big) \:
\frac{h^{1/2}}{\epsilon} \; |M|_g^{1/2}.
  \end{aligned}
$$
The terms $E_3$ and $E_4$ are estimated in the same way that we
have already done, that is
$$
  \begin{aligned}
  & E_3 \leq C \: |M|_g
  \Big(1+ \|u_0\|_{L^\infty}\Big) \: \frac{h}{\epsilon},
  \\
  & E_4 \leq C \: T \: |M|_g
  \Big(1+ \|u_0\|_{L^\infty}\Big) \:
  \epsilon.
  \end{aligned}
$$
Finally, we estimate the last term, that is
$$
  \begin{aligned}
  E_5&= \iint_{Q_T}\int_{0}^{T} \sum_{K \in \TT^h} |K| \: \frac{p_K}{|K|} \: \dashint_{\partial K}
  \big|F_{y'}(u_{K}^n,c)-F_{y}(u_{K}^n,c)\big|
\\
  & \hspace{80pt} 
  \big|\tilde{\psi}'_{\epsilon}(x;x')-\psi_{\epsilon}(x;y')\big| \;
  \tilde{\rho}'_\delta(t;t') \; d\Gamma_g(y') dt'  dv_g(x)dt
\\
  & \leq C \:T \: |M|_g
\Big(1+ \|u_0\|_{L^\infty}\Big) \:
\frac{h}{\epsilon},
  \end{aligned}
$$
where we have used the condition \eqref{IN2}.
\end{proof}

\section*{Acknowledgements}

The authors were partially supported by the A.N.R. (Agence
Nationale de la Recherche) through the grant 06-2-134423 entitled
{\sl ``Mathematical Methods in General Relativity''} (MATH-GR),
and by the Centre National de la Recherche Scientifique (CNRS).
The second author (WN) was also supported by the grant 311759/2006-8 
from the National Counsel of Technological and Scientific Development (CNPq)  
and by an international cooperation project between Brazil and France.

\newcommand{\auth}{\textsc}


\end{document}